\documentclass{amsart}%
\usepackage{eurosym}
\usepackage{amsmath}
\usepackage{amssymb}
\usepackage{amsfonts}
\usepackage{graphicx}%
\newtheorem{theorem}{Theorem}[section]
\theoremstyle{plain}

\newtheorem{corollary}[theorem]{Corollary}

\newtheorem{example}{Example}

\newtheorem{lemma}[theorem]{Lemma}

\newtheorem{proposition}[theorem]{Proposition}
\newtheorem{remark}[theorem]{Remark}

\numberwithin{equation}{section}
\begin{document}
\title[Multiple solutions to a nonlinear Choquard equation]{Positive and sign changing solutions to a nonlinear Choquard equation}
\author{M\'{o}nica Clapp}
\address{Instituto de Matem\'{a}ticas, Universidad Nacional Aut\'{o}noma de M\'{e}xico,
Circuito Exterior, C.U., 04510 M\'{e}xico D.F., Mexico.\smallskip}
\email{M\'{o}nica Clapp $<$mclapp@matem.unam.mx$>$}
\author{Dora Salazar}
\email{Dora Salazar $<$docesalo@gmail.com$>$}
\thanks{Research supported by CONACYT grant 129847 and UNAM-DGAPA-PAPIIT grant
IN106612 (Mexico).}
\date{September 24, 2012}

\begin{abstract}
\noindent We consider the problem%
\[
-\Delta u+W(x)u=\left(  \frac{1}{|x|^{\alpha}}\ast|u|^{p}\right)
|u|^{p-2}u,\text{ \ \ }u\in H_{0}^{1}(\Omega),
\]
where $\Omega$ is an exterior domain in $\mathbb{R}^{N}$, $N\geq3,$ $\alpha
\in(0,N)$, $p\in\lbrack2,\frac{2N-\alpha}{N-2}),$ $W\in\emph{C}^{\,0}%
(\mathbb{R}^{N})$, $\inf_{\mathbb{R}^{N}}W>0,$ and $W(x)\rightarrow V_{\infty
}>0$ as $\left\vert x\right\vert \rightarrow\infty$. Under symmetry
assumptions on $\Omega$ and $W,$ which allow finite symmetries, and some
assumptions on the decay of $W$ at infinity, we establish the existence of a
positive solution and multiple sign changing solutions to this problem, having
small energy.\smallskip\newline\textsc{Key words: }Nonlinear Choquard
equation; nonlocal nonlinearity; exterior domain; positive and sign changing
solutions. \noindent\newline\textsc{MSC2010: }Primary 35J91. Secondary 35A01,
35B06, 35J20, 35Q55.

\end{abstract}
\maketitle

\section{Introduction}

We consider the problem
\begin{equation}
\left\{
\begin{array}
[c]{l}%
-\Delta u+\left(  V_{\infty}+V(x)\right)  u=\left(  \frac{1}{|x|^{\alpha}}%
\ast|u|^{p}\right)  |u|^{p-2}u,\\
u\in H_{0}^{1}(\Omega),
\end{array}
\right. \label{prob}%
\end{equation}
where $N\geq3$, $\alpha\in(0,N)$, $p\in\bigl(\frac{2N-\alpha}{N}%
,\frac{2N-\alpha}{N-2}\bigr)$ and $\Omega$ is an unbounded smooth domain in
$\mathbb{R}^{N}$ whose complement $\mathbb{R}^{N}\smallsetminus\Omega$ is
bounded, possibly empty. We also assume that the potential $V_{\infty}+V$ satisfies

\medskip

\noindent$\mathbf{(}V_{0}\mathbf{)}$ $\ \ V\in\mathcal{C}^{0}(\mathbb{R}^{N})
$, $\ V_{\infty}\in(0,\infty)$, $\ \inf_{x\in\mathbb{R}^{N}}\{V_{\infty
}+V(x)\}>0$, $\lim_{|x|\rightarrow\infty}V(x)=0$.

\medskip A special case of (\ref{prob}), relevant in physical applications, is
the Choquard equation
\begin{equation}
-\Delta u+u=\left(  \frac{1}{|x|}\ast|u|^{2}\right)  u,\quad u\in
H^{1}(\mathbb{R}^{3}),\label{choq}%
\end{equation}
which models an electron trapped in its own hole, and was proposed by Choquard
in 1976 as an approximation to Hartree-Fock theory of a one-component plasma
\cite{lb}. This equation arises in many interesting situations related to the
quantum theory of large systems of nonrelativistic bosonic atoms and
molecules, see for example \cite{f,ls} and the references therein. It was also
proposed by Penrose in 1996 as a model for the self-gravitational collapse of
a quantum mechanical wave-function \cite{p}. In this context, problem
(\ref{choq}) is usually called the nonlinear Schr\"{o}dinger-Newton equation,
see also \cite{mp,mt}.

In 1976 Lieb \cite{lb} proved the existence and uniqueness (modulo
translations) of a minimizer to problem (\ref{choq}) by using symmetric
decreasing rearrangement inequalities. Later, in \cite{ln}, Lions showed the
existence of infinitely many radially symmetric solutions to (\ref{choq}).
Further results for related problems may be found in
\cite{a,ccs2,css,m,n,s,ww} and the references therein.

In 2010, Ma and Zhao \cite{mz} considered the generalized Choquard equation
\begin{equation}
-\Delta u+u=\left(  \frac{1}{|x|^{\alpha}}\ast|u|^{p}\right)  |u|^{p-2}u,\quad
u\in H^{1}(\mathbb{R}^{N}),\label{gchoq}%
\end{equation}
and proved that, for $p\geq2$, every positive solution of it is radially
symmetric and monotone decreasing about some point, under the assumption that
a certain set of real numbers, defined in terms of $N,$ $\alpha$ and $p,$ is
nonempty. Under the same assumption, Cingolani, Clapp and Secchi \cite{ccs}
recently gave some existence and multiplicity results in the electromagnetic
case, and established the regularity and some decay asymptotics at infinity of
the ground states of (\ref{gchoq}). Moroz and van Schaftingen \cite{ms}
eliminated this restriction and showed the regularity, positivity and radial
symmetry of the ground states for the optimal range of parameters, and derived
decay asymptotics at infinity for them, as well. These results will play an
important role in our study.

In this article, we are interested in obtaining positive and sign changing
solutions to problem (\ref{prob}). We study the case where both $\Omega$ and
$V$ have some symmetries. If $\Gamma$ is a closed subgroup of the group $O(N)$
of linear isometries of $\mathbb{R}^{N},$ we denote by $\Gamma x:=\{gx:g\in
\Gamma\}$ the $\Gamma$-orbit of $x$, by $\#\Gamma x$ its cardinality, and by
\[
\ell(\Gamma):=\min\{\#\Gamma x:x\in\mathbb{R}^{N}\smallsetminus\{0\}\}.
\]
We assume that $\Omega$ and $V$ are $\Gamma$-invariant, this means that
$\Gamma x\subset\Omega$ for every $x\in\Omega$ and that $V$ is constant on
$\Gamma x$ for each $x\in\mathbb{R}^{N}.$ We consider a continuous group
homomorphism $\phi:\Gamma\rightarrow\mathbb{Z}/2$ and we look for solutions
which satisfy
\begin{equation}
u(gx)=\phi(g)u(x)\text{\hspace{0.3in}for all }g\in\Gamma\text{ and }x\in
\Omega.\label{equi}%
\end{equation}
A function $u$ with this property will be called $\phi$\emph{-equivariant.} We
denote by
\[
G:=\ker\phi.
\]
Note that, if $u$ satisfies (\ref{equi}), then $u$ is $G$-invariant. Moreover,
$u(\gamma x)=-u(x)$ for every $x\in\Omega$ and $\gamma\in\phi^{-1}(-1).$
Therefore, if $\phi$ is an epimorphism (i.e. if it is surjective), every
nontrivial solution to (\ref{prob}) which satisfies (\ref{equi})\ changes
sign. If $\phi\equiv1$ is the trivial homomorphism, then $\Gamma=G$ and
(\ref{equi}) simply says that $u$ is $G$-invariant.

If $Z$ is a $\Gamma$-invariant subset of $\mathbb{R}^{N}$ and $\phi$ is an
epimorphism, the group $\mathbb{Z}/2$ acts on the $G$-orbit space
$Z/G:=\{Gx:x\in Z\}$ of $Z$ as follows: we choose $\gamma\in\Gamma$ such that
$\phi(\gamma)=-1$ and we define
\[
(-1)\cdot Gx:=G(\gamma x)\text{\hspace{0.3in}for all }x\in Z.
\]
This action is well defined and it does not depend on the choice of $\gamma.$
We denote by
\[
\Sigma:=\{x\in\mathbb{R}^{N}:\left\vert x\right\vert =1\text{, }\#\Gamma
x=\ell(\Gamma)\},\text{\qquad}\Sigma_{0}:=\{x\in\Sigma:Gx=G(\gamma x)\}.
\]
If $Z$ is a nonempty $\Gamma$-invariant subset of $\Sigma\smallsetminus
\Sigma_{0}$, the action of $\mathbb{Z}/2$ on its $G$-orbit space $Z/G$ is free
and the \emph{Krasnoselskii genus} of $Z/G$, denoted genus$(Z/G)$, is defined
to be the smallest $k\in\mathbb{N}$ such that there exists a continuous map
$f:Z/G\rightarrow\mathbb{S}^{k-1}:=\{x\in\mathbb{R}^{k}:\left\vert
x\right\vert =1\}$ which is $\mathbb{Z}/2$-equivariant, i.e. $f((-1)\cdot
Gz)=-f(Gz)$ for every $z\in Z.$ We define genus$(\emptyset):=0.$

For each subgroup $K$ of $O(N)$ and each $K$-invariant subset $Z$ of
$\mathbb{R}^{N}\smallsetminus\{0\}$ we set
\[
\mu(Kz):=\left\{
\begin{array}
[c]{ll}%
\inf\{\left\vert gz-hz\right\vert :g,h\in K,\text{ }gz\neq hz\} & \text{if
}\#Kz\geq2,\\
2\left\vert z\right\vert  & \text{if }\#Kz=1,
\end{array}
\right.
\]%
\[
\mu_{K}(Z):=\inf_{z\in Z}\mu(Kz)\qquad\text{and}\qquad\mu^{K}(Z):=\sup_{z\in
Z}\mu(Kz).
\]
In the special case where $K=G$ and $Z=\Sigma$, we simply write
\[
\mu_{G}:=\mu_{G}(\Sigma)\text{\qquad and\qquad}\mu^{G}:=\mu^{G}(\Sigma).
\]
We only consider the case $\ell(\Gamma)<\infty$, because if all $\Gamma
$-orbits of $\Omega$ are infinite it was already shown in \cite[Theorem
1.1]{ccs} that (\ref{prob}) has infinitely many solutions. In this case,
$\mu_{G}>0.$

We denote by $c_{\infty}$ the energy of a ground state of the problem
\[
\left\{
\begin{array}
[c]{l}%
-\Delta u+V_{\infty}u=\left(  \frac{1}{|x|^{\alpha}}\ast|u|^{p}\right)
|u|^{p-2}u,\\
u\in H^{1}(\mathbb{R}^{N}).
\end{array}
\right.
\]
We shall look for solutions with small energy, i.e. which satisfy
\begin{equation}
\frac{p-1}{2p}\int_{\mathbb{R}^{N}}\int_{\mathbb{R}^{N}}\frac{|u(x)|^{p}%
|u(y)|^{p}}{|x-y|^{\alpha}}dx\,dy<\ell(\Gamma)c_{\infty}.\label{ener}%
\end{equation}
In what follows, we assume that $V$ satisfies $(V_{0})$ and we consider two
cases: the case in which $V$ is strictly negative at infinity, and that in
which $V$ takes on nonnegative values at infinity (which includes the case
$V=0$). We shall prove the following results:

\begin{theorem}
\label{slna1}If $p\geq2$, $\Omega$ is $G$-invariant and $V$ is a $G$-invariant
function which satisfies

\noindent$\mathbf{(}V_{1}\mathbf{)}$ There exist $r_{0}>0$, $c_{0}>0$ and
$\lambda\in(0,\mu^{G}\sqrt{V_{\infty}})$ such that
\[
V(x)\leq-c_{0}e^{-\lambda|x|}\text{\hspace{0.3in}for all }\,x\in\mathbb{R}%
^{N}\text{ with }\left\vert x\right\vert \geq r_{0},
\]
then \emph{(\ref{prob})} has at least one positive solution $u$ which is
$G$-invariant and satisfies \emph{(\ref{ener})} with $\Gamma=G.$
\end{theorem}

\begin{theorem}
\label{slna2}If $p\geq2$, $\Omega$ is $\Gamma$-invariant, $\phi:\Gamma
\rightarrow\mathbb{Z}/2$ is an epimorphism, $Z$ is a $\Gamma$-invariant subset
of $\Sigma\smallsetminus\Sigma_{0},$ $V$ is a $\Gamma$-invariant function and
the following holds:

\noindent$\mathbf{(}V_{2}\mathbf{)}$ \ There exist $r_{0}>0$, $c_{0}>0$ and
$\lambda\in(0,\mu_{\Gamma}(Z)\sqrt{V_{\infty}})$ such that
\[
V(x)\leq-c_{0}e^{-\lambda|x|}\text{\hspace{0.3in}for all }\,x\in\mathbb{R}%
^{N}\text{ with }\left\vert x\right\vert \geq r_{0},
\]
then problem \emph{(\ref{prob})} has at least \emph{genus}$(Z/G)$ pairs of
sign changing solutions $\pm u$, which satisfy \emph{(\ref{equi})} and
\emph{(\ref{ener})}.
\end{theorem}

\begin{theorem}
\label{slna3}If $p\geq2$, $\ell(G)\geq3$, $\Omega$ is $G$-invariant and $V$ is
a $G$-invariant function which satisfies

\noindent$\mathbf{(}V_{3}\mathbf{)}$ There exist $c_{0}>0$ and $\kappa>\mu
_{G}\sqrt{V_{\infty}}$ such that
\[
V(x)\leq c_{0}e^{-\kappa|x|}\hspace{0.3in}\text{for all }\,x\in\mathbb{R}^{N},
\]
then \emph{(\ref{prob})} has at least one positive solution $u$ which is
$G$-invariant and satisfies \emph{(\ref{ener})} with $\Gamma=G.$
\end{theorem}

\begin{theorem}
\label{slna4}If $p\geq2$, $\Omega$ is $\Gamma$-invariant, $\phi:\Gamma
\rightarrow\mathbb{Z}/2$ is an epimorphism, $Z$ is a $\Gamma$-invariant subset
of $\Sigma$, $V$ is a $\Gamma$-invariant function and the following hold:

\noindent$(Z_{0})$ $\ $There exists $a_{0}>1$ such that
\[
\text{\emph{dist}}(\gamma z,Gz)\geq a_{0}\mu(Gz)\qquad\text{for all }z\in
Z\text{\ and\ }\gamma\in\Gamma\smallsetminus G.
\]
$\mathbf{(}V_{4}\mathbf{)}$ \ There exist $c_{0}>0$ and $\kappa>\mu^{\Gamma
}(Z)\sqrt{V_{\infty}}$ such that
\[
V(x)\leq c_{0}e^{-\kappa|x|}\hspace{0.3in}\text{for all }\,x\in\mathbb{R}^{N},
\]
then \emph{(\ref{prob})} has at least \emph{genus}$(Z/G)$ pairs of sign
changing solutions $\pm u$, which satisfy \emph{(\ref{equi})} and
\emph{(\ref{ener})}.
\end{theorem}

Theorem \ref{slna1} was proved in \cite{ccs} for $\Omega=\mathbb{R}^{N},$
under additional assumptions on $\alpha$ and $p.$ As far as we know, Theorem
\ref{slna3} is the first existence result for potentials $V$ which are
nontrivial and take nonnegative values at infinity. In the local case, Bahri
and Lions proved existence for this type of potentials without any symmetries
\cite{bl}. Unfortunately, some of the facts used in their proof are not
available in the nonlocal case.

As we mentioned before, the existence of infinitely many solutions is known in
the radial case \cite{ln} and in the case where every $G$-orbit in $\Omega$ is
infinite \cite{ccs}. In contrast, Theorems \ref{slna2} and \ref{slna4} provide
multiple solutions when the data have only finite symmetries. The following
examples, which illustrate these results, are taken from \cite{cs}, where
similar results for the local case were recently obtained.

\begin{example}
\emph{Let }$\Gamma$\emph{\ be the group spanned by the reflection }%
$\gamma:\mathbb{R}^{N}\rightarrow\mathbb{R}^{N}$\emph{\ on a linear subspace
}$W$\emph{\ of }$\mathbb{R}^{N}$\emph{. If }$\Omega$\emph{\ and }%
$V$\emph{\ are invariant under this reflection, we may take }$\phi
:\Gamma\rightarrow\mathbb{Z}/2$\emph{\ to be the epimorphism given by }%
$\phi(\gamma):=-1$\emph{\ and }$Z$\emph{\ to be the unit sphere in the
orthogonal complement of }$W$\emph{. Then, Theorem \ref{slna2} yields }%
\[
\text{\emph{genus}}(Z)=N-\dim W
\]
\emph{pairs of solutions to problem (\ref{prob}) provided }$(V_{2}%
)$\emph{\ holds for some }$\lambda\in(0,2\sqrt{V_{\infty}})$\emph{. }
\end{example}

\begin{example}
\emph{If }$N=2n$\emph{\ we identify} $\mathbb{R}^{N}$ \emph{with}
$\mathbb{C}^{n}$ \emph{and take }$\Gamma$\emph{\ to be the cyclic group of
order }$2m$\emph{\ spanned by }$\rho(z_{1},\ldots,z_{n}):=(e^{\pi i/m}%
z_{1},\ldots,e^{\pi i/m}z_{n})$\emph{\ and} $\phi:\Gamma\rightarrow
\mathbb{Z}/2$ \emph{to be the epimorphism given by }$\phi(\rho):=-1.$%
\emph{\ Then }$G:=\ker\phi$\emph{\ is the cyclic subgroup of order }%
$m$\emph{\ spanned by} $\rho^{2},$ $\Sigma=\mathbb{S}^{N-1}$ \emph{and}
$\Sigma_{0}=\emptyset.$ \emph{So we may take} $Z:=\mathbb{S}^{N-1}.$ \emph{The
genus of} $\mathbb{S}^{N-1}/G$ \emph{can be estimated in many cases. For
example, if }$m=2^{k}$\emph{, Lemma 6.1 in \cite{cs} together with Theorem 1.2
in \cite{b} give }%
\[
\text{\emph{genus}}(\mathbb{S}^{N-1}/G)\geq\frac{N-1}{2^{k}}+1.
\]
\emph{Since }$\mu_{\Gamma}(\mathbb{S}^{N-1})=\left\vert e^{\pi i/m}%
-1\right\vert ,$\emph{\ if condition }$(V_{2})$\emph{\ holds for }$m=2^{k}%
,$\emph{\ it will also hold for }$m=2^{j}$\emph{\ with }$0\leq j<k.$\emph{\ An
easy argument shows that, if }$u_{j}$\emph{\ satisfies (\ref{equi}) for
}$m=2^{j},$\emph{\ }$u_{l}$\emph{\ satisfies (\ref{equi}) for }$m=2^{l}
$\emph{\ and }$j\neq l,$\emph{\ then }$u_{j}\neq u_{l}$\emph{, see
\cite[section 1]{cs}. Therefore, Theorem \ref{slna2} provides at least}%
\[
\sum_{j=0}^{k}\frac{N-1}{2^{j}}+k+1=(N-1)\frac{2^{k+1}-1}{2^{k}}+k+1
\]
\emph{pairs of sign changing solutions in this case.}
\end{example}

The group $G$ in the previous example satisfies $\ell(G)=m.$ This shows that
there are many groups satisfying the symmetry assumption in Theorem
\ref{slna3} when $N$ is even. If $N$ is odd not many groups satisfy
$\ell(G)\geq3.$ For example, if $N=3,$ the only subgroups of $O(3)$ which
satisfy this condition are the rotation groups of the icosahedron, octahedron
and tetrahedron, $I$, $O$ and $T$, and the groups $I\times\mathbb{Z}_{2}^{c}$,
$O\times\mathbb{Z}_{2}^{c}$, $T\times\mathbb{Z}_{2}^{c}$ and $O^{-}$ described
in \cite[Appendix A]{clm}.

Note that $(Z_{0})$ implies that $Z\subset\Sigma\smallsetminus\Sigma_{0}$.
Condition $(Z_{0})$ cannot be realized if $N=3$. Next, we give an example for
which $(Z_{0})$ holds.

\begin{example}
\emph{We identify} $\mathbb{R}^{4n}$ \emph{with} $\mathbb{C}^{n}%
\times\mathbb{C}^{n}$ \emph{and consider the subgroup }$\Gamma$\emph{\ of
}$O(4n)$\emph{\ spanned by }$\rho$\emph{\ and }$\gamma,$\emph{\ where }%
$\rho(y,z):=(e^{\pi i/m}y,e^{\pi i/m}z)$\emph{\ and }$\gamma(y,z):=(-\overline
{z},\overline{y})$\emph{\ for} $(y,z)\in\mathbb{C}^{n}\times\mathbb{C}^{n}$
\emph{and some }$m\geq3.$\emph{\ We define }$\phi:\Gamma\rightarrow
\mathbb{Z}/2$\emph{\ by }$\phi(\rho)=1,$\emph{\ }$\phi(\gamma)=-1.$%
\emph{\ Then }$G:=\ker\phi$\emph{\ is the cyclic subgroup of order }%
$2m$\emph{\ spanned by }$\rho.$\emph{\ Since }$m\geq3,$\emph{\ property
}$(Z_{0})$\emph{\ holds for} $Z:=\mathbb{S}^{4n-1}.$ \emph{We showed in
\cite[Proposition 6.1]{cs} that genus}$(\mathbb{S}^{4n-1}/G)\geq2n+1.$
\emph{Consequently, if }$\Omega$\emph{\ and }$V$\emph{\ are }$\Gamma
$\emph{-invariant and }$(V_{4})$\emph{\ holds, Theorem \ref{slna4} yields
}$2n+1$\emph{\ pairs of sign changing solutions to problem (\ref{prob}). Note
that} $\mu^{G}(\mathbb{S}^{4n-1})=\left\vert e^{\pi i/m}-1\right\vert ,$
\emph{hence }$(V_{4})$\emph{\ becomes less restrictive as }$m$%
\emph{\ increases.}
\end{example}

This paper is organized as follows: In section 2 we set the variational
framework for problem (\ref{prob}). In section 3 some preliminary asymptotic
estimates are established. In section 4 we consider potentials which are
strictly negative at infinity and prove Theorems \ref{slna1} and \ref{slna2}.
Finally, in section 5 we consider potentials which take on nonnegative values
at infinity and prove Theorems \ref{slna3} and \ref{slna4}.

\section{The variational setting}

From now on we shall assume without loss of generality that $V_{\infty}=1$.
Assumption $(V_{0})$ guarantees that
\begin{equation}
\left\langle u,v\right\rangle _{V}:=\int_{\Omega}\nabla u\cdot\nabla
v+\int_{\Omega}\left(  1+V(x)\right)  uv\label{pe}%
\end{equation}
is a scalar product in $H_{0}^{1}(\Omega)$ and that the induced norm
\begin{equation}
\left\Vert u\right\Vert _{V}:=\left(  \int_{\Omega}\left(  \left\vert \nabla
u\right\vert ^{2}+\left(  1+V(x)\right)  u^{2}\right)  \right)  ^{1/2}%
\label{normV}%
\end{equation}
is equivalent to the usual one. If $V=0$ we write $\left\langle
u,v\right\rangle $ and $\left\Vert u\right\Vert $ instead of $\left\langle
u,v\right\rangle _{0}$ and $\left\Vert u\right\Vert _{0}.$

As usual, we identify $u\in H_{0}^{1}(\Omega)$ with its extension to
$\mathbb{R}^{N}$ obtained by setting $u=0$ in $\mathbb{R}^{N}\smallsetminus
\Omega$. We define
\[
\mathbb{D}(u):=\int_{\Omega}\left(  \frac{1}{|x|^{\alpha}}\ast|u|^{p}\right)
|u|^{p}=\int_{\mathbb{R}^{N}}\int_{\mathbb{R}^{N}}\frac{|u(x)|^{p}|u(y)|^{p}%
}{|x-y|^{\alpha}}dx\,dy
\]
and set $r:=\frac{2N}{2N-\alpha}$. As $p\in(\frac{2N-\alpha}{N},\frac
{2N-\alpha}{N-2})$, one has that $pr\in(2,\frac{2N}{N-2})$. The
Hardy-Littlewood-Sobolev inequality \cite[Theorem 4.3]{ll} implies the
existence of a positive constant $\bar{C}$ such that
\begin{equation}
\mathbb{D}(u)\leq\bar{C}|u|_{pr}^{2p}\text{\qquad for all \ }u\in
H^{1}(\mathbb{R}^{N}),\label{desigualdad}%
\end{equation}
where $\left\vert u\right\vert _{q}:=\left(  \int_{\mathbb{R}^{N}}\left\vert
u\right\vert ^{q}\right)  ^{1/q}$ is the norm in $L^{q}(\mathbb{R}^{N}).$ This
shows that $\mathbb{D}$ is well defined in $H^{1}(\mathbb{R}^{N})$.

We shall assume from now on that $p\in\bigl[2,\frac{2N-\alpha}{N-2}\bigr)$.
Then the functional $J_{V}:H_{0}^{1}(\Omega)\rightarrow\mathbb{R}$ given by
\begin{equation}
J_{V}(u):=\frac{1}{2}\left\Vert u\right\Vert _{V}^{2}-\frac{1}{2p}%
\mathbb{D}(u)\label{J}%
\end{equation}
is of class $\mathcal{C}^{2}$. Its derivative is
\[
J_{V}^{\prime}(u)v:=\left\langle u,v\right\rangle _{V}-\int_{\Omega}\left(
\frac{1}{|x|^{\alpha}}\ast|u|^{p}\right)  |u|^{p-2}uv.
\]
Hence, the solutions to problem (\ref{prob}) are the critical points of
$J_{V}$.

The homomorphism $\phi:\Gamma\rightarrow\mathbb{Z}/2$ induces an orthogonal
action of $\Gamma$ on $H_{0}^{1}(\Omega)$ as follows: for $\gamma\in\Gamma$
and $u\in H_{0}^{1}(\Omega)$ we define $\gamma u\in H_{0}^{1}(\Omega)$ by
\[
(\gamma u)(x):=\phi(\gamma)u(\gamma^{-1}x).
\]
Since $\left\langle \gamma u,\gamma v\right\rangle _{V}=\left\langle
u,v\right\rangle _{V}$ and $\mathbb{D}(\gamma u)=\mathbb{D}(u)$ for all
$\gamma\in\Gamma$, $u,v\in H_{0}^{1}(\Omega)$, the functional $J_{V}$ is
$\Gamma$-invariant. By the principle of symmetric criticality \cite{Palais,
Willem} the critical points of the restriction of $J_{V}$ to the fixed point
space of this action, which we denote by
\begin{align*}
H_{0}^{1}(\Omega)^{\phi}: &  =\{u\in H_{0}^{1}(\Omega):\gamma u=u\,\ \forall
\gamma\in\Gamma\}\\
&  =\{u\in H_{0}^{1}(\Omega):u(\gamma x)=\phi(\gamma)u(x)\,\ \forall\gamma
\in\Gamma,\,\forall x\in\Omega\},
\end{align*}
are the solutions to problem (\ref{prob}) that satisfy (\ref{equi}). The
nontrivial ones lie on the Nehari manifold
\[
\mathcal{N}_{\Omega,V}^{\phi}:=\left\{  u\in H_{0}^{1}(\Omega)^{\phi}%
:u\neq0,\,\Vert u\Vert_{V}^{2}=\mathbb{D}(u)\right\}  ,
\]
which is of class $\mathcal{C}^{2}$ and radially diffeomorphic to the unit
sphere in $H_{0}^{1}(\Omega)^{\phi}$. The radial projection $\pi:H_{0}%
^{1}(\Omega)^{\phi}\smallsetminus\{0\}\rightarrow\mathcal{N}_{\Omega,V}^{\phi
}$ is given by
\begin{equation}
\pi(u):=\left(  \dfrac{\Vert u\Vert_{V}^{2}}{\mathbb{D}(u)}\right)  ^{\frac
{1}{2(p-1)}}u.\label{max}%
\end{equation}
Accordingly, for every $u\in H_{0}^{1}(\Omega)^{\phi}\smallsetminus\{0\},$%
\begin{equation}
J_{V}(\pi(u))=\frac{p-1}{2p}\left(  \dfrac{\Vert u\Vert_{V}^{2}}%
{\mathbb{D}(u)^{\frac{1}{p}}}\right)  ^{\frac{p}{p-1}}.\label{minmax}%
\end{equation}
We set
\[
c_{\Omega,V}^{\phi}:=\inf_{\mathcal{N}_{\Omega,V}^{\phi}}J_{V}.
\]
If $\phi\equiv1$ is the trivial homomorphism, then $\Gamma=G:=\ker\phi.$ In
this case we shall write $H_{0}^{1}(\Omega)^{G}$, $\mathcal{N}_{\Omega,V}^{G}$
and $c_{\Omega,V}^{G}\ $instead of $H_{0}^{1}(\Omega)^{\phi}$, $\mathcal{N}%
_{\Omega,V}^{\phi}$ and $c_{\Omega,V}^{\phi}$. If $G=\{1\}$ is the trivial
group, we shall omit it from the notation and write simply $H_{0}^{1}(\Omega
)$, $\mathcal{N}_{\Omega,V}$ and $c_{\Omega,V}$.

The problem
\begin{equation}
\left\{
\begin{array}
[c]{l}%
-\Delta u+u=\left(  \frac{1}{|x|^{\alpha}}\ast|u|^{p}\right)  |u|^{p-2}u,\\
u\in H^{1}(\mathbb{R}^{N}),
\end{array}
\right. \label{problim}%
\end{equation}
plays a special role: it is the limit problem for (\ref{prob}). In this case
we write $J_{\infty},$ $\mathcal{N}_{\infty}$ and $c_{\infty}\ $instead of
$J_{0},$ $\mathcal{N}_{\mathbb{R}^{N},0}$ and $c_{\mathbb{R}^{N},0}.$

It is known that $c_{\infty}$ is attained at a positive function $\omega\in
H^{1}(\mathbb{R}^{N})$ (see for example \cite[Theorem 3]{ms}). The following
result shows, however, that $c_{\Omega,V}^{\phi}$ is not necessarily attained.
We write
\[
B_{r}(\xi):=\{x\in\mathbb{R}^{N}:\left\vert x-\xi\right\vert <r\}.
\]

\begin{proposition}
If $V\geq0$, then $c_{\Omega,V}=c_{\infty}.$ If, additionally, $V\not \equiv $
$0$ when $\Omega=\mathbb{R}^{N}$, then $c_{\Omega,V}$ is not attained.
\end{proposition}

\begin{proof}
Since $H_{0}^{1}(\Omega)\subset H^{1}(\mathbb{R}^{N})$ and $V\geq0$ one easily
concludes that $c_{\Omega,V}\geq c_{\infty}.$ Let $R>0$ be such that $\left(
\mathbb{R}^{N}\smallsetminus\Omega\right)  \subset B_{R}(0)$, and let
$(x_{n})$ be a sequence in $\mathbb{R}^{N}$ such that $\left\vert
x_{n}\right\vert >R$ and $\left\vert x_{n}\right\vert \rightarrow\infty$. We
choose a cut-off function $\chi\in\mathcal{C}_{c}^{\infty}(\mathbb{R}^{N})$
such that $0\leq\chi\leq1,$ $\chi(x)=1$ if $\left\vert x\right\vert \leq1$ and
$\chi(x)=0$ if $\left\vert x\right\vert \geq2.$ We define $r_{n}:=\frac{1}%
{2}(\left\vert x_{n}\right\vert -R)$ and
\[
u_{n}(x):=\chi\Bigl(\frac{x-x_{n}}{r_{n}}\Bigr)\omega(x-x_{n}).
\]
Then $u_{n}\in H_{0}^{1}(\Omega)$, $u_{n}\neq0$, $u_{n}\rightharpoonup0$
weakly in $H_{0}^{1}(\mathbb{R}^{N})$ and $u_{n}\rightarrow0$ strongly in
$L_{loc}^{2}(\mathbb{R}^{N})$. An easy argument shows that
\[
\lim_{n\rightarrow\infty}\Vert u_{n}\Vert_{V}^{2}=\lim_{n\rightarrow\infty
}\Vert u_{n}\Vert^{2}=\Vert\omega\Vert^{2}\text{.}%
\]
Applying the Lebesgue dominated convergence theorem, we obtain that
\[
\lim_{n\rightarrow\infty}\mathbb{D}(u_{n})=\mathbb{D}(\omega).
\]
Consequently, from (\ref{minmax})\ we obtain that $J_{V}(\pi(u_{n}%
))\rightarrow J_{\infty}(\omega)=c_{\infty}$. Therefore $c_{\Omega,V}\leq
c_{\infty},$ and hence $c_{\Omega,V}=c_{\infty}.$

\noindent Now, if there were $u\in\mathcal{N}_{\Omega,V}$ satisfying
$J_{V}(u)=c_{\Omega,V}$, then $u$ would be a nontrivial solution of problem
(\ref{problim}) with minimum energy and $\Vert u\Vert_{V}^{2}=\Vert u\Vert
^{2}.$ We distinguish two cases: (1) If $\Omega=\mathbb{R}^{N}$ then, by
assumption, $V$ is strictly positive on some open set $U$ of $\mathbb{R}^{N}$.
Since
\[
0=\Vert u\Vert_{V}^{2}-\Vert u\Vert^{2}=\int_{\mathbb{R}^{N}}V(x)u^{2}\geq
\int_{U}V(x)u^{2}\geq0,
\]
we conclude that $u=0$ in $U.$ (2) If $\Omega\neq\mathbb{R}^{N}$ then $u=0$ in
$\mathbb{R}^{N}\smallsetminus\Omega.$ In both cases, we obtain a contradiction
to the unique continuation principle \cite{gl, jk}. As a result, $c_{\Omega
,V}$ is not attained.
\end{proof}

We say that $J_{V}$ \emph{satisfies condition }$(PS)_{c}^{\phi}$ if every
sequence $(u_{n})$ such that
\begin{equation}
u_{n}\in H_{0}^{1}(\Omega)^{\phi},\hspace{0.3in}J_{V}(u_{n})\rightarrow
c,\hspace{0.3in}J_{V}^{\prime}(u_{n})\rightarrow0\,\text{ in }\,H^{-1}%
(\Omega),\label{ps}%
\end{equation}
has a convergent subsequence in $H_{0}^{1}(\Omega)$. If $\phi\equiv1$, we
write $(PS)_{c}^{G}$ instead of $(PS)_{c}^{\phi}.$

\begin{proposition}
\label{propps}$J_{V}$ satisfies condition $(PS)_{c}^{\phi}$ for all
\[
c<\ell(\Gamma)c_{\infty}.
\]

\end{proposition}

\begin{proof}
This follows from Proposition 3.1 in \cite{ccs} taking $A=0,$ $G=\Gamma$,
$\tau=\phi$ (notice that $\mathbb{Z}/2$ is a subgroup of $\mathbb{S}^{1}$) and
$u_{n}\in H_{0}^{1}(\Omega)^{\phi}\subset H^{1}(\mathbb{R}^{N},\mathbb{C}%
)^{\phi}.$
\end{proof}

We denote by $\nabla J_{V}$ the gradient of $J_{V}$ with respect to the scalar
product (\ref{pe}), and by $\nabla_{\mathcal{N}}J_{V}(u)$ the orthogonal
projection of $\nabla J_{V}(u)$ onto the tangent space $T_{u}\mathcal{N}%
_{\Omega,V}^{\phi}$ to the Nehari manifold $\mathcal{N}_{\Omega,V}^{\phi}$ at
the point $u\in\mathcal{N}_{\Omega,V}^{\phi}.$ We shall say that $J_{V}$
\emph{satisfies condition} $(PS)_{c}^{\phi}$ \emph{on} $\mathcal{N}_{\Omega
,V}^{\phi}$ if every sequence $(u_{n})$ such that
\begin{equation}
u_{n}\in\mathcal{N}_{\Omega,V}^{\phi},\hspace{0.3in}J_{V}(u_{n})\rightarrow
c,\hspace{0.3in}\nabla_{\mathcal{N}}J_{V}(u_{n})\rightarrow0,\label{psN}%
\end{equation}
contains a convergent subsequence in $H_{0}^{1}(\Omega)$.

\begin{corollary}
\label{corps}$J_{V}$ satisfies condition$\ (PS)_{c}^{\phi}$ on $\mathcal{N}%
_{\Omega,V}^{\phi}$ for all
\[
c<\ell(\Gamma)c_{\infty}.
\]

\end{corollary}

\begin{proof}
The proof is completely analogous to that of Corollary 3.8 in \cite{cs}.
\end{proof}

\section{Asymptotic estimates}

The ground states of problem (\ref{problim}) have been recently studied in
\cite{ccs, ms}. The following result holds true.

\begin{theorem}
\label{asinf} Let $\omega$ be a ground state of problem \emph{(\ref{problim}%
)}. Then $\omega\in L^{1}(\mathbb{R}^{N})\cap\mathcal{C}^{\infty}%
(\mathbb{R}^{N})$, $\omega$ does not change sign and it is radially symmetric
and monotone decreasing in the radial direction with respect to some fixed
point. Moreover, $\omega$ has the following asymptotic behavior:

\begin{enumerate}
\item[(i)] If $p>2$ then
\[
\lim\limits_{|x|\rightarrow\infty}|\omega(x)||x|^{\frac{N-1}{2}}e^{|x|}%
\in(0,\infty).
\]

\item[(ii)] If $p=2$ then
\[
\lim\limits_{|x|\rightarrow\infty}|\omega(x)||x|^{\frac{N-1}{2}}e^{Q(|x|)}%
\in(0,\infty),
\]
where
\[
Q(t):=\int_{\delta}^{t}\sqrt{1-\frac{\delta^{\alpha}}{s^{\alpha}}}%
ds\quad\text{and}\quad\delta^{\alpha}:=(4-\alpha)c_{\infty}.
\]

\end{enumerate}
\end{theorem}

\begin{proof}
See Theorems 3 and 4 in \cite{ms}. Note that $\omega$ is a solution of
(\ref{problim}) if and only if $u:=\lambda^{-\frac{1}{2(p-1)}}\omega$ is a
solution of problem $(1.1)$ in \cite{ms}, where $\lambda:=\frac{\Gamma
(\alpha/2)}{\Gamma((N-\alpha)/2)\pi^{N/2}2^{N-\alpha}}$ and $\Gamma$ denotes
here (and only here) the gamma function (and not the group).
\end{proof}

In what follows, $\omega$ will denote a positive ground state of problem
(\ref{problim}) which is radially symmetric with respect to the origin. We
continue to assume that $p\geq2$.

\begin{lemma}
\label{propasin}
\[
\lim\limits_{|x|\rightarrow\infty}\omega(x)|x|^{\frac{N-1}{2}}e^{a|x|}=
\begin{cases}
\infty & \text{if $a>1$},\\
0 & \text{if $a\in(0,1).$}%
\end{cases}
\]

\end{lemma}

\begin{proof}
Set $b:=\frac{N-1}{2}$. We shall prove this result for $p=2$. The proof for
$p>2$ is an immediate consequence of Theorem \ref{asinf}. Observe that, for
every $\nu\in(0,1)$ it holds true that
\[
\sqrt{1-\frac{\delta^{\alpha}}{s^{\alpha}}}\leq1\quad\text{if }s\geq
\delta\qquad\text{and}\qquad\sqrt{1-\frac{\delta^{\alpha}}{s^{\alpha}}}\geq
\nu\quad\text{if }s\geq\frac{\delta}{(1-\nu^{2})^{1/\alpha}}=:s_{\nu},
\]
and, hence, that
\[
Q(t)\leq t\quad\text{if }t\geq\delta\qquad\text{and}\qquad\nu(t-s_{\nu})\leq
Q(t)\quad\text{if }t\geq s_{\nu}.
\]
Consequently, if $|x|\geq\delta$ then
\[
\omega(x)|x|^{b}e^{a|x|}=\omega(x)|x|^{b}e^{Q(|x|)}e^{a|x|-Q(|x|)}\geq
\omega(x)|x|^{b}e^{Q(|x|)}e^{(a-1)|x|}.
\]
If $a>1$, the conclusion follows from Theorem \ref{asinf}. If $a\in(0,1)$, we
fix $\nu\in(a,1)$. Then, for all $|x|\geq s_{\nu}$,
\[
\omega(x)|x|^{b}e^{a|x|}=\omega(x)|x|^{b}e^{Q(|x|)}e^{a|x|-Q(|x|)}\leq
\omega(x)|x|^{b}e^{Q(|x|)}e^{(a-\nu)|x|+\nu s_{\nu}},
\]
and using once more Theorem \ref{asinf} the conclusion follows.
\end{proof}

For $\zeta\in\mathbb{R}^{N}$ we set\
\begin{equation}
\omega_{\zeta}(x):=\omega(x-\zeta).\label{omegaz}%
\end{equation}

\begin{lemma}
\label{a<1}For each $a\in(0,1),$%
\[
\lim_{\left\vert \zeta\right\vert \rightarrow\infty}\int_{\mathbb{R}^{N}%
}\omega^{p-1}\omega_{\zeta}|\zeta|^{\frac{N-1}{2}}e^{a|\zeta|}=0.
\]

\end{lemma}

\begin{proof}
By Lemma \ref{propasin}\ we have that, for each $\nu\in(0,1),$ there exists a
constant $C_{\nu}>0$ such that
\[
\omega(x)\leq C_{\nu}e^{-\nu|x|}\qquad\text{for all }x\in\mathbb{R}^{N}.
\]
We fix $\nu_{1},\nu_{2}\in(a,1)$ with $\nu_{1}<\nu_{2}.$ In what follows, $C$
will denote different positive constants depending only on $\nu_{1}$ and
$\nu_{2}$. We have that
\begin{align*}
\int_{\mathbb{R}^{N}}\omega^{p-1}\omega_{\zeta} &  \leq C\int_{\mathbb{R}^{N}%
}e^{-\nu_{1}(p-1)|x|}e^{-\nu_{2}|x-\zeta|}\text{ }dx\leq C\int_{\mathbb{R}%
^{N}}e^{-\nu_{1}|x|}e^{-\nu_{2}|x-\zeta|}\text{ }dx\\
&  =C\int_{\mathbb{R}^{N}}e^{-\nu_{1}(|x|+|x-\zeta|)}e^{-(\nu_{2}-\nu
_{1})|x-\zeta|}\text{ }dx\leq Ce^{-\nu_{1}|\zeta|}\int_{\mathbb{R}^{N}%
}e^{-(\nu_{2}-\nu_{1})|x|}\text{ }dx\\
&  =Ce^{-\nu_{1}|\zeta|}.
\end{align*}
Therefore,
\[
0\leq\int_{\mathbb{R}^{N}}\omega^{p-1}\omega_{\zeta}|\zeta|^{\frac{N-1}{2}%
}e^{a|\zeta|}\leq C|\zeta|^{\frac{N-1}{2}}e^{-(\nu_{1}-a)|\zeta|},
\]
which implies the result.
\end{proof}

For $\zeta\in\mathbb{R}^{N}$ we define
\begin{equation}
I(\zeta):=\int_{\mathbb{R}^{N}}\left(  \frac{1}{|x|^{\alpha}}\ast\omega
^{p}\right)  \omega^{p-1}\omega_{\zeta}.\label{main}%
\end{equation}

\begin{lemma}
\label{Ia<1}For each $a\in(0,1),$
\[
\lim_{\left\vert \zeta\right\vert \rightarrow\infty}I(\zeta)\left\vert
\zeta\right\vert ^{\frac{N-1}{2}}e^{a\left\vert \zeta\right\vert }=0.
\]

\end{lemma}

\begin{proof}
As $p<\frac{2N-\alpha}{N-2},$ we have that $\frac{N-\alpha}{N}\bigl(1-\frac
{1}{p}\bigr)-\frac{2}{N}<\frac{N-\alpha}{Np}$. By \cite[Section 4, Claim
1]{ms}, $|u|^{p}\in L^{\frac{N}{N-\alpha}}(\mathbb{R}^{N})$. Hence, $\frac
{1}{|x|^{\alpha}}\ast\omega^{p}\in L^{\infty}(\mathbb{R}^{N})$, cf.
\cite[Section 4.3 (9)]{ll}. Thus,
\[
0\leq I(\zeta)|\zeta|^{\frac{N-1}{2}}e^{a|\zeta|}\leq C\int_{\mathbb{R}^{N}%
}\omega^{p-1}\omega_{\zeta}|\zeta|^{\frac{N-1}{2}}e^{a|\zeta|}.
\]
From Lemma \ref{a<1} we obtain the conclusion.
\end{proof}

\begin{lemma}
\label{inf2}For every $a>1$, there exists a positive constant $k_{a}$ such
that
\[
I(\zeta)|\zeta|^{\frac{N-1}{2}}e^{a|\zeta|}\geq k_{a}\qquad\text{for all
}\left\vert \zeta\right\vert \geq1.
\]

\end{lemma}

\begin{proof}
Set $b:=\frac{N-1}{2}.$ Lemma \ref{propasin} asserts the existence of positive
constants $C_{a}$, $R_{a}$ such that $C_{a}|x|^{-b}e^{-a|x|}\leq\omega(x)$ if
$\left\vert x\right\vert \geq R_{a}.$ Let $C_{1}>0$ be such that
$\omega(x)\geq C_{1}e^{-a|x|}$ for all $\left\vert x\right\vert \leq R_{a}$.
Setting $C_{2}:=\min\{C_{a},C_{1}\}$ we conclude that
\[
\omega(x)\geq C_{2}(1+|x|)^{-b}e^{-a|x|}\text{\quad for all \ }x\in
\mathbb{R}^{N}.
\]
Hence,
\begin{align*}
\omega(x-\zeta)|\zeta|^{b}e^{a|\zeta|} &  \geq C_{2}(1+|x-\zeta|)^{-b}%
e^{-a|x-\zeta|}|\zeta|^{b}e^{a|\zeta|}\\
&  \geq C_{2}(1+|x-\zeta|)^{-b}|\zeta|^{b}e^{-a|x|}\qquad\text{for }x,\zeta
\in\mathbb{R}^{N}.
\end{align*}
Note that, if $\left\vert x\right\vert \leq1\leq\left\vert \zeta\right\vert $,
then $1+\left\vert x-\zeta\right\vert \leq1+\left\vert x\right\vert
+\left\vert \zeta\right\vert \leq3\left\vert \zeta\right\vert $ and so
\[
\omega(x-\zeta)|\zeta|^{b}e^{a|\zeta|}\geq C_{3}e^{-a|x|}\quad\text{for
}x,\zeta\in\mathbb{R}^{N}\text{ with }\left\vert x\right\vert \leq
1\leq\left\vert \zeta\right\vert ,
\]
where $C_{3}:=3^{-b}C_{2}$. Consequently,
\begin{align*}
I(\zeta)|\zeta|^{b}e^{a|\zeta|} &  =\int_{\mathbb{R}^{N}}\left(  \frac
{1}{|x|^{\alpha}}\ast\omega^{p}\right)  (x)\,\omega^{p-1}(x)\omega
(x-\zeta)|\zeta|^{b}e^{a|\zeta|}\,dx\\
&  \geq C_{3}\int_{|x|\leq1}\left(  \frac{1}{|x|^{\alpha}}\ast\omega
^{p}\right)  (x)\,\omega^{p-1}(x)e^{-a|x|}=:k_{a}\qquad\text{for }\left\vert
\zeta\right\vert \geq1,
\end{align*}
as claimed.
\end{proof}

\begin{remark}
\emph{As in the local case (see \cite[section 5]{cs}) it is possible to prove
that, for }$p>2,$\emph{\ there exists a positive constant }$k_{1}$\emph{\ such
that}
\[
\lim_{|\xi|\rightarrow\infty}I(\xi)|\xi|^{\frac{N-1}{2}}e^{|\xi|}=k_{1}.
\]
\emph{However, we will not need this fact.}
\end{remark}

For $\zeta\in\mathbb{R}^{N}$ we define
\begin{equation}
A(\zeta):=\int_{\mathbb{R}^{N}}V^{+}(x)\omega^{2}(x-\zeta)dx.\label{A}%
\end{equation}

\begin{lemma}
\label{apositiva}Let $M\in(0,2).$ If $V(x)\leq ce^{-\iota\left\vert
x\right\vert }$ for all $x\in\mathbb{R}^{N}$ with $c>0$ and $\iota>M,$ then
\[
\lim\limits_{\left\vert \zeta\right\vert \rightarrow\infty}A(\zeta)\left\vert
\zeta\right\vert ^{\frac{N-1}{2}}e^{M\left\vert \zeta\right\vert }=0.
\]

\end{lemma}

\begin{proof}
See \cite[Lemma 5.2]{cs}.
\end{proof}

\begin{lemma}
\label{h}If $f\in\mathcal{C}_{c}^{0}\left(  \mathbb{R}^{N}\right)  ,$ $q>1$
and $a\in(0,1)$, then
\[
\lim\limits_{\left\vert \zeta\right\vert \rightarrow\infty}\left(
\int_{\mathbb{R}^{N}}f(x)\omega^{q}(x-\zeta)dx\right)  \left\vert
\zeta\right\vert ^{\frac{N-1}{2}}e^{qa\left\vert \zeta\right\vert }=0.
\]

\end{lemma}

\begin{proof}
Set $b:=\frac{N-1}{2}$. Let $T>0$ be such that supp$(f)\subset B_{T}(0).$ By
Lemma \ref{propasin} there exists $C>0$ such that
\[
\omega(x)\leq C(T+|x|)^{-b}e^{-a|x|}\text{\qquad for all }x\in\mathbb{R}^{N}.
\]
Therefore, if $\left\vert x\right\vert \leq T,$%
\begin{align*}
\omega^{q}(x-\zeta)\left\vert \zeta\right\vert ^{b}e^{qa\left\vert
\zeta\right\vert }  &  \leq C^{q}(T+|x-\zeta|)^{-qb}e^{-qa|x-\zeta|}\left\vert
\zeta\right\vert ^{b}e^{qa\left\vert \zeta\right\vert }\\
&  \leq C^{q}(\left\vert x\right\vert +|x-\zeta|)^{-qb}e^{-qa|x-\zeta
|}\left\vert \zeta\right\vert ^{b}e^{qa\left\vert \zeta\right\vert }\leq
C^{q}\left\vert \zeta\right\vert ^{(1-q)b}e^{qa\left\vert x\right\vert }.
\end{align*}
Consequently,
\[
\int_{\mathbb{R}^{N}}\left\vert f(x)\right\vert \omega^{q}(x-\zeta)\left\vert
\zeta\right\vert ^{b}e^{qa\left\vert \zeta\right\vert }dx\leq C^{q}\left\vert
\zeta\right\vert ^{(1-q)b}\int_{\left\vert x\right\vert \leq T}\left\vert
f(x)\right\vert e^{qa\left\vert x\right\vert }dx=:C_{1}\left\vert
\zeta\right\vert ^{(1-q)b},
\]
from which the assertion of Lemma \ref{h} follows.
\end{proof}

\section{Proof of Theorems \ref{slna1} and \ref{slna2}}

Let $Z$ be a $\Gamma$-invariant subset of $\Sigma$ and let $\lambda\in
(0,\mu_{\Gamma}(Z))$ be such that $(V_{2})$ holds (recall that we are assuming
that $V_{\infty}=1$). We choose $\nu\in(0,1)$ such that $\lambda\in
(0,\mu_{\Gamma}(Z)\nu)$, $\ \varepsilon\in\bigl(0,\frac{\mu_{\Gamma}%
(Z)\nu-\lambda}{\mu_{\Gamma}(Z)\nu+\lambda}\bigr)$ \ and a radially symmetric
cut-off function $\chi\in\mathcal{C}^{\infty}(\mathbb{R}^{N})$ such that
$0\leq\chi\leq1,$ $\chi(x)=1$ if $\left\vert x\right\vert \leq1-\varepsilon$
and $\chi(x)=0$ if $\left\vert x\right\vert \geq1.$ Let $\omega\in
H^{1}(\mathbb{R}^{N})$ be a positive ground state of problem (\ref{problim})
which is radially symmetric about the origin. For $S>0$ we define $\omega
^{S}\in H^{1}(\mathbb{R}^{N})$ by
\[
\omega^{S}(x):=\chi\left(  \frac{x}{S}\right)  \omega(x).
\]
Lemma \ref{propasin} allows to obtain the following asymptotic estimates:
\[
\left\vert \left\Vert \omega\right\Vert ^{2}-\left\Vert \omega^{S}\right\Vert
^{2}\right\vert =O\bigl(e^{-2\nu(1-\varepsilon)S}\bigr),\text{\hspace{0.3in}%
}\left\vert \mathbb{D}(\omega)-\mathbb{D}(\omega^{S})\right\vert
=O\bigl(e^{-p\nu(1-\varepsilon)S}\bigr)
\]
as $S\rightarrow\infty$, see \cite[Lemma 4.1]{ccs}. We set $\rho:=\frac
{\mu_{\Gamma}(Z)\nu+\lambda}{4\nu},$ and for every $z\in Z$ we consider the
function
\[
v_{R,z}(x):=\omega^{\rho R}(x-Rz).
\]
Note that \ supp$(v_{R,z})\subset\overline{B_{\rho R}(Rz)}$. Note also that
$\rho\in(0,1)$ because $\mu_{\Gamma}(Z)\leq2$. Therefore, since $\mathbb{R}%
^{N}\smallsetminus\Omega$ is bounded, there exists $R_{0}>0$ such that
$v_{R,z}\in H_{0}^{1}(\Omega)$ for all $z\in Z$ and $R\geq R_{0}.$

\begin{lemma}
\label{estim1}There exist $d_{0}>0$ and $\varrho_{0}>R_{0}$ such that
$v_{R,z}\in H_{0}^{1}(\Omega)$ and
\[
J_{V}(\pi(v_{R,z}))\leq c_{\infty}-d_{0}e^{-\lambda R}\hspace{0.3in}\text{for
all }z\in Z\text{ and\ }R\geq\varrho_{0}.
\]

\end{lemma}

\begin{proof}
This is a special case of \cite[Lemma 4.2]{ccs}\ with $A=0.$
\end{proof}

Let $\phi:\Gamma\rightarrow\mathbb{Z}/2$ be a continuous group homomorphism
and set $G:=\ker\phi$. We fix $R\geq\varrho_{0},$ and for $z\in Z$ we define
\begin{equation}
\theta(z):={\ {\textstyle\sum\limits_{gz\in\Gamma z}} }\phi(g)v_{R,gz}%
.\label{theta}%
\end{equation}

\begin{proposition}
\label{estimps}If either $\phi\equiv1$ or $Z\subset\Sigma\smallsetminus
\Sigma_{0}$, then $\theta(z)$ is well defined. $\theta(z)$ is $\phi
$-equivariant and
\[
J_{V}(\pi(\theta(z)))\leq\ell(\Gamma)\left(  c_{\infty}-d_{0}e^{-\lambda
R}\right)  \hspace{0.3in}\text{for all }z\in Z.
\]
If moreover $Z\neq\emptyset$, then$\ c_{\Omega,V}^{\phi}<\ell(\Gamma
)c_{\infty}.$
\end{proposition}

\begin{proof}
Let $z\in Z.$ If $g_{1},g_{2}\in\Gamma$ are such that $g_{1}z=g_{2}z$, then
$g_{2}^{-1}g_{1}z=z.$ Hence, if either $\phi\equiv1$ or $z\notin\Sigma_{0},$
it must be true that $\phi(g_{2}^{-1}g_{1})=1.$ Thus $\phi(g_{1})=\phi
(g_{2}).$ This shows that $\theta(z)$ is well defined. It is clearly $\phi$-equivariant.

\noindent On the other hand, since $\left\vert Rg_{1}z-Rg_{2}z\right\vert \geq
R\mu_{\Gamma}(Z)>2\rho R$ when $g_{1}z\neq g_{2}z,$ we have that
supp$(v_{R,g_{1}z})\cap$ supp$(v_{R,g_{2}z})=\emptyset.$ Consequently,
$\Vert\theta(z)\Vert_{V}^{2}=\ell(\Gamma)\Vert v_{R,z}\Vert_{V}^{2}$
and$\ \mathbb{D}(\theta(z))>\ell(\Gamma)\mathbb{D}(v_{R,z}).$ From
(\ref{minmax}) and Lemma \ref{estim1} we obtain
\begin{align*}
J_{V}(\pi(\theta(z)))  &  \leq\frac{p-1}{2p}\left(  \dfrac{\ell(\Gamma)\Vert
v_{R,z}\Vert_{V}^{2}}{\left[  \ell(\Gamma)\mathbb{D}(v_{R,z})\right]
^{\frac{1}{p}}}\right)  ^{\frac{p}{p-1}}\\
&  =\ell(\Gamma)J_{V}(\pi(v_{R,z}))\leq\ell(\Gamma)\left(  c_{\infty}%
-d_{0}e^{-\lambda R}\right)  .
\end{align*}
Finally, since $\pi(\theta(z))\in\mathcal{N}_{\Omega,V}^{\phi},$ we conclude
that $c_{\Omega,V}^{\phi}<\ell(\Gamma)c_{\infty}.$
\end{proof}

\smallskip

\begin{proof}
[Proof of Theorem \ref{slna1}]Let $\phi\equiv1,$ so that $\Gamma=G.$ If
assumption $(V_{1})$ holds for $\lambda\in(0,\mu^{G}),$ we choose $\zeta
\in\Sigma$ such that $\mu(G\zeta)\in(\lambda,\mu^{G}]$ and define $Z:=G\zeta.$
Thus $\mu_{G}(Z)=\mu(G\zeta)$ and assumption $(V_{2})$ holds for $\lambda
\in(0,\mu_{G}(Z)).$ Hence, we may apply Proposition \ref{estimps} to these
data to conclude that $c_{\Omega,V}^{G}<\ell(G)c_{\infty}.$ Corollary
\ref{corps}\ then asserts that $J_{V}$ satisfies condition $(PS)_{c}^{G}$ on
$\mathcal{N}_{\Omega,V}^{G}$ for $c:=c_{\Omega,V}^{G}.$ Therefore, there
exists $u\in\mathcal{N}_{\Omega,V}^{G}$ such that $J_{V}(u)=c_{\Omega,V}^{G}.$
Finally, observe that $\left\vert u\right\vert \in\mathcal{N}_{\Omega,V}^{G}$
and $J_{V}(\left\vert u\right\vert )=J_{V}(u).$ Hence problem (\ref{prob}) has
a $G$-invariant positive solution $\left\vert u\right\vert $ satisfying
$J_{V}(\left\vert u\right\vert )<\ell(G)c_{\infty}.$
\end{proof}

\smallskip

\begin{proof}
[Proof of Theorem \ref{slna2}]$\mathcal{N}_{\Omega,V}^{\phi}$ is a
$\mathcal{C}^{2}$-manifold and $J_{V}:\mathcal{N}_{\Omega,V}^{\phi}%
\rightarrow\mathbb{R}$ is an even $\mathcal{C}^{2}$-function, which is bounded
from below and satisfies $(PS)_{c}^{\phi}$ on $\mathcal{N}_{\Omega,V}^{\phi}$
for all $c<\ell(\Gamma)c_{\infty}$. Therefore, if $d:=\ell(\Gamma)\left(
c_{\infty}-d_{0}e^{-\lambda R}\right)  ,$ then$\ J_{V}$ has at least
\[
\text{genus}(\mathcal{N}_{\Omega,V}^{\phi}\cap J_{V}^{d})
\]
pairs of critical points $\pm u$ with $J_{V}(u)\leq d,$ where $J_{V}%
^{d}:=\{u\in H_{0}^{1}(\Omega):J_{V}(u)\leq d\}.$

\noindent The map $\theta:Z\rightarrow\mathcal{N}_{\Omega,V}^{\phi}\cap
J_{V}^{d}$ defined by (\ref{theta}) is continuous. Furthermore, $\theta
(gz)=\theta(z)$ for all $g\in G$\ and\ $\theta(\gamma z)=-\theta(z)$ if
$\phi(\gamma)=-1.$ Consequently, $\theta$ induces a continuous map
$\widehat{\theta}:Z/G\rightarrow\mathcal{N}_{\Omega,V}^{\phi}\cap J_{V}^{d},$
given by $\widehat{\theta}(Gz):=\theta(z),$ which satisfies $\widehat{\theta
}((-1)\cdot Gz)=-\widehat{\theta}(Gz)$ for all $z\in Z.$ This implies that
\[
\text{genus}(Z/G)\leq\text{ genus}(\mathcal{N}_{\Omega,V}^{\phi}\cap J_{V}%
^{d})
\]
and concludes the proof.
\end{proof}

\section{Proof of Theorems \ref{slna3} and \ref{slna4}}

Let $\phi:\Gamma\rightarrow\mathbb{Z}/2$ be a continuous group homomorphism
and set $G:=\ker\phi.$ Let $\omega\in H^{1}(\mathbb{R}^{N})$ be a positive
ground state of problem (\ref{problim}) which is radially symmetric about the
origin, and let $Z$ be a nonempty $\Gamma$-invariant subset of $\Sigma.$ If
$\phi$ is an epimorphism, we also assume that $Z\subset\Sigma\smallsetminus
\Sigma_{0}$. Thus, for $z\in Z$ and $R>0,$ the function
\begin{equation}
\sigma_{Rz}:={\textstyle\sum\limits_{gz\in\Gamma z}}\phi(g)\omega_{Rgz}%
,\quad\text{ where \ }\omega_{\zeta}(x):=\omega(x-\zeta),\label{sigma}%
\end{equation}
is well defined and $\phi$-equivariant (see Proposition \ref{estimps}). In
addition, we assume that \smallskip

\noindent$(Z_{\ast})$ $\ \mu^{\Gamma}(Z)<2$ \ and \ there exists $a_{0}>1$
such that
\[
\text{dist}(\gamma z,Gz)\geq a_{0}\mu(Gz)\quad\text{for any }z\in
Z\text{\ and\ }\gamma\in\Gamma\smallsetminus G.
\]

We choose $R_{0}>0$ such that $\left(  \mathbb{R}^{N}\smallsetminus
\Omega\right)  \subset B_{R_{0}}(0),$ and a radially symmetric cut-off
function $\chi\in\mathcal{C}^{\infty}(\mathbb{R}^{N})$ such that $0\leq
\chi(x)\leq1,$ $\chi(x)=0$ if $\left\vert x\right\vert \leq R_{0}$ and
$\chi(x)=1$ if $\left\vert x\right\vert \geq2R_{0}$. Observe that $\chi
\sigma_{R}\in H_{0}^{1}(\Omega)^{\phi}$. We shall prove the following result.

\begin{proposition}
\label{sub}If $Z$ and $V$ satisfy $(Z_{\ast})$ and $(V_{4})\ $then there exist
$c_{0},R_{0}>0$ and $\beta>1$ such that
\begin{equation}
\dfrac{\Vert\chi\sigma_{Rz}\Vert_{V}^{2}}{\mathbb{D}(\chi\sigma_{Rz}%
)^{\frac{1}{p}}}\leq\bigl(\ell(\Gamma)\left\Vert \omega\right\Vert
^{2}\bigr)^{\frac{p-1}{p}}-c_{0}e^{-\beta R}\quad\text{for any \ }R\geq
R_{0},\text{ }z\in Z.\label{lhs}%
\end{equation}
Consequently, $c_{\Omega,V}^{\phi}<\ell(\Gamma)c_{\infty}$.
\end{proposition}

We require some preliminary lemmas.

\begin{lemma}
\label{l0}\emph{(i)} If $p\geq2$ and $a_{1},\ldots,a_{n}\geq0,$ then
\[
\left\vert {\ {\textstyle\sum\limits_{i=1}^{n}} }a_{i}\right\vert ^{p}%
\geq{{\textstyle\sum\limits_{i=1}^{n}} }a_{i}^{p}+(p-1){\ {\textstyle\sum
\limits_{i\neq k}} }a_{i}^{p-1}a_{k}.
\]
\emph{(ii)} If $p\geq2$ and $a,b\geq0,$ then
\[
\left\vert a-b\right\vert ^{p}\geq a^{p}+b^{p}-p\left(  a^{p-1}b+ab^{p-1}%
\right)  .
\]

\end{lemma}

\begin{proof}
See Lemma 4 in \cite{cc}.
\end{proof}

\begin{lemma}
\label{l1}If $p\geq2,$ $A={{\textstyle\sum\limits_{i=1}^{n}} }a_{i},$
$\tilde{A}={\ {\textstyle\sum\limits_{i=1}^{n}} }\tilde{a}_{i},$
$B={\ {\textstyle\sum\limits_{i=1}^{n}} }b_{i}$ \text{ and }$\tilde
{B}={\ {\textstyle\sum\limits_{i=1}^{n}} }\tilde{b}_{i}$ with $a_{i},\tilde
{a}_{i},b_{i},\tilde{b}_{i}\geq0,$ then
\begin{align}
{A}^{p}{B}^{p}  &  \geq{\ {\textstyle\sum\limits_{i=1}^{n}} }a_{i}^{p}%
b_{i}^{p}+(p-1)\left(  {\textstyle\sum\limits_{j\neq m}} a_{j}^{p}b_{j}%
^{p-1}b_{m}+{\ {\textstyle\sum\limits_{i\neq k}} }b_{i}^{p}a_{i}^{p-1}%
a_{k}\right)  ,\label{1}\\
{A}^{2}{B}^{2}  &  \geq{\ {\textstyle\sum\limits_{i=1}^{n}} }a_{i}^{2}%
b_{i}^{2}+2\left(  {\textstyle\sum\limits_{j\neq m}} a_{j}^{2}b_{j}%
b_{m}+{\ {\textstyle\sum\limits_{i\neq k}} }b_{i}^{2}a_{i}a_{k}\right)
,\label{2}\\
\left\vert A-\tilde{A}\right\vert ^{p}\left\vert B-\tilde{B}\right\vert ^{p}
&  \geq{A}^{p}{B}^{p}+\tilde{A}^{p}\tilde{B}^{p}\label{3}\\
&  -pn^{p-1}\left(  B^{p}+\tilde{B}^{p}\right)  \left[  \left(
{\ {\textstyle\sum\limits_{i=1}^{n}} }a_{i}^{p-1}\right)  \tilde{A}+\left(
{\ {\textstyle\sum\limits_{i=1}^{n}} }\tilde{a}_{i}^{p-1}\right)  A\right]
\nonumber\\
&  -pn^{p-1}\left(  A^{p}+\tilde{A}^{p}\right)  \left[  \left(
{\ {\textstyle\sum\limits_{i=1}^{n}} }b_{i}^{p-1}\right)  \tilde{B}+\left(
{\ {\textstyle\sum\limits_{i=1}^{n}} }\tilde{b}_{i}^{p-1}\right)  B\right]
.\nonumber
\end{align}

\end{lemma}

\begin{proof}
Using Lemma \ref{l0}(i) we obtain
\begin{align*}
&  \left\vert {\ {\textstyle\sum\limits_{i=1}^{n}} }a_{i}\right\vert
^{p}\left\vert {\ {\textstyle\sum\limits_{j=1}^{n}} }b_{j}\right\vert ^{p}%
\geq\left(  {\ {\textstyle\sum\limits_{i=1}^{n}} }a_{i}^{p}%
+(p-1){\ {\textstyle\sum\limits_{i\neq k}} }a_{i}^{p-1}a_{k}\right)  \left(
{\ {\textstyle\sum\limits_{j=1}^{n}} }b_{j}^{p}+(p-1){\ {\textstyle\sum
\limits_{j\neq m}} }b_{j}^{p-1}b_{m}\right) \\
&  \geq{{\textstyle\sum\limits_{i=1}^{n}} }a_{i}^{p}b_{i}^{p}+(p-1)
{\textstyle\sum\limits_{j\neq m}} (a_{j}^{p}+a_{m}^{p})b_{j}^{p-1}%
b_{m}+(p-1){{\textstyle\sum\limits_{i\neq k}} }\left(  b_{i}^{p}+b_{k}%
^{p}\right)  a_{i}^{p-1}a_{k}.
\end{align*}
Inequalities (\ref{1}) and (\ref{2}) can be immediately deduced from the above expression.

\noindent On the other hand, applying Lemma \ref{l0}(ii) we obtain
\begin{align*}
&  \left\vert A-\tilde{A}\right\vert ^{p}\left\vert B-\tilde{B}\right\vert
^{p}\\
&  \geq\left[  A^{p}+\tilde{A}^{p}-p\bigl(A^{p-1}\tilde{A}+A\tilde{A}%
^{p-1}\bigr)\right]  \left[  B^{p}+\tilde{B}^{p}-p\bigl( B^{p-1}\tilde
{B}+B\tilde{B}^{p-1}\bigl) \right] \\
&  \geq{A}^{p}{B}^{p}+\tilde{A}^{p}\tilde{B}^{p}-p\bigl(B^{p}+\tilde{B}%
^{p}\bigr)\bigl(A^{p-1}\tilde{A}+A\tilde{A}^{p-1}\bigr)-p\bigl(A^{p}+\tilde
{A}^{p}\bigr)\bigl(B^{p-1}\tilde{B}+B\tilde{B}^{p-1}\bigr),
\end{align*}
which yields inequality (\ref{3}).
\end{proof}

\begin{lemma}
\label{descorte}For every $u\in H^{1}(\mathbb{R}^{N})$ the following
inequalities hold:
\begin{align*}
\Vert\chi u\Vert_{V}^{2} &  \leq\Vert u\Vert_{V}^{2}-\int_{\mathbb{R}^{N}%
}(\chi\Delta\chi)u^{2},\\
\mathbb{D}(\chi u) &  \geq\mathbb{D}(u)-2\int_{\mathbb{R}^{N}}\int
_{\mathbb{R}^{N}}\frac{(1-\chi^{p}(x))|u(x)|^{p}|u(y)|^{p}}{|x-y|^{\alpha}%
}dx\,dy.
\end{align*}

\end{lemma}

\begin{proof}
For every $u\in H^{1}(\mathbb{R}^{N})$ one has that
\begin{align*}
\Vert\chi u\Vert_{V}^{2} &  =\int_{\mathbb{R}^{N}}\left(  \left\vert
\chi\nabla u+u\nabla\chi\right\vert ^{2}+(1+V(x))\left\vert \chi u\right\vert
^{2}\right) \\
&  =\int_{\mathbb{R}^{N}}\chi^{2}\left(  |\nabla u|^{2}+(1+V(x))|u|^{2}%
\right)  +\int_{\mathbb{R}^{N}}\bigl(|\nabla\chi|^{2}-\frac{1}{2}\Delta
(\chi^{2})\bigr)u^{2}\\
&  \leq\Vert u\Vert_{V}^{2}-\int_{\mathbb{R}^{N}}(\chi\Delta\chi)u^{2}.
\end{align*}
Writing $ab=1-(1-a)-(1-b)+(1-a)(1-b)$ and taking $a:=\chi^{p}(x),$
$b:=\chi^{p}(y),$ we obtain
\begin{align*}
\mathbb{D}(\chi u) &  =\int_{\mathbb{R}^{N}}\int_{\mathbb{R}^{N}}\frac
{\chi^{p}(x)\chi^{p}(y)|u(x)|^{p}|u(y)|^{p}}{|x-y|^{\alpha}}dx\,dy\\
&  =\mathbb{D}(u)-2\int_{\mathbb{R}^{N}}\int_{\mathbb{R}^{N}}\frac{(1-\chi
^{p}(x))|u(x)|^{p}|u(y)|^{p}}{|x-y|^{\alpha}}dx\,dy\\
&  +\int_{\mathbb{R}^{N}}\int_{\mathbb{R}^{N}}\frac{(1-\chi^{p}(x))(1-\chi
^{p}(y))|u(x)|^{p}|u(y)|^{p}}{|x-y|^{\alpha}}dx\,dy.
\end{align*}
Notice that the last summand in the right-hand side of the above expression is
nonnegative. Then the second inequality follows.
\end{proof}

We shall apply this lemma to the function $\sigma_{Rz}$ to derive inequality
(\ref{lhs}). To this purpose we also require some asymptotic estimates, which
will be provided by the following four lemmas.

Since $\omega$ is a solution of problem (\ref{problim}), for any $z,z^{\prime
}\in\mathbb{R}^{N}$, one has that $J_{\infty}^{\prime}(\omega_{z}%
)\omega_{z^{\prime}}=0$, which is equivalent to
\[
\int_{\mathbb{R}^{N}}{\left[  \nabla\omega_{z}\cdot\nabla\omega_{z^{\prime}%
}+\omega_{z}\omega_{z^{\prime}}\right]  }=\int_{\mathbb{R}^{N}}\left(
\frac{1}{|x|^{\alpha}}\ast\omega_{z}^{p}\right)  \omega_{z}^{p-1}%
\omega_{z^{\prime}}.
\]
A change of variable in the right-hand side of this inequality allows us to
express it as
\begin{equation}
\left\langle \omega_{z},\omega_{z^{\prime}}\right\rangle =I(z^{\prime
}-z)\text{\qquad for all \ }z,z^{\prime}\in\mathbb{R}^{N},\label{wI}%
\end{equation}
where $\left\langle \cdot,\cdot\right\rangle $ is the usual scalar product in
$H^{1}(\mathbb{R}^{N})$ and $I$ is the function defined in (\ref{main}). We
denote by $Fz:=\{\left(  gz,hz\right)  \in\Gamma z\times\Gamma z:gz\neq hz\}$
and define
\begin{align*}
\varepsilon_{Rz} &  :={\textstyle\sum\limits_{\substack{\left(  gz,hz\right)
\in Fz \\\phi(g)=\phi(h)}}}I(Rgz-Rhz),\text{ }\\
\widehat{\varepsilon}_{Rz} &  :={\textstyle\sum\limits_{_{\substack{\left(
gz,hz\right)  \in Fz \\\phi(g)\neq\phi(h)}}}}I(Rgz-Rhz)\text{ \ if }%
\phi\not \equiv 1,\text{\quad and\quad}\widehat{\varepsilon}_{Rz}:=0\text{
\ if }\phi\equiv1.
\end{align*}
We choose $g_{z},h_{z}\in Gz$ such that $\left\vert g_{z}z-h_{z}z\right\vert
=\mu(\Gamma z):=\min\{\left\vert gz-hz\right\vert :g,h\in\Gamma,$ $gz\neq
hz\}$ and set
\[
\xi_{z}:=g_{z}z-h_{z}z.
\]

\begin{lemma}
\label{epsilongorro}If $(Z_{\ast})$ holds, then
\[
\widehat{\varepsilon}_{Rz}=o(\varepsilon_{Rz})
\]
uniformly in $z\in Z.$
\end{lemma}

\begin{proof}
For $a_{0}>1$ as in condition $(Z_{\ast})$\ we fix $\widehat{a}\in(0,1)$ such
that $a:=\widehat{a}a_{0}>1.$ Thus, $a\left\vert \xi_{z}\right\vert
=a\mu(Gz)\leq\widehat{a}\left\vert gz-hz\right\vert $ for any $z\in Z,$
$g,h\in\Gamma$ with $gz\neq hz$ and $\phi(g)\neq\phi(h).$ Lemma \ref{inf2}
yields a constant $k_{a}>0$ such that
\[
I(R\xi_{z})|R\xi_{z}|^{b}e^{a|R\xi_{z}|}\geq k_{a}\qquad\text{if \ }R\geq
\mu_{\Gamma}(Z)^{-1},
\]
where $b:=\frac{N-1}{2}.$ So, setting $C:=k_{a}^{-1}$ we obtain
\begin{align*}
\frac{I(Rgz-Rhz)}{I(R\xi_{z})} &  \leq\frac{I(Rgz-Rhz)\left\vert
Rgz-Rhz\right\vert ^{b}e^{\widehat{a}\left\vert Rgz-Rhz\right\vert }}%
{I(R\xi_{z})|R\xi_{z}|^{b}e^{a|R\xi_{z}|}}\\
&  \leq CI(Rgz-Rhz)\left\vert Rgz-Rhz\right\vert ^{b}e^{\widehat{a}\left\vert
Rgz-Rhz\right\vert }\qquad\text{if \ }R\geq\mu_{\Gamma}(Z)^{-1}.
\end{align*}
Let $\varepsilon>0.$ Lemma \ref{Ia<1} asserts that there exists $S>0$ such
that $I(\zeta)\left\vert \zeta\right\vert ^{b}e^{\widehat{a}\left\vert
\zeta\right\vert }<\varepsilon$ if $\left\vert \zeta\right\vert >S.$ As
$\widehat{a}\left\vert Rgz-Rhz\right\vert \geq Ra\mu_{G}>0,$ taking
$R_{0}:=\max\{\frac{\widehat{a}S}{a\mu_{G}},\mu_{\Gamma}(Z)^{-1}\}$ we
conclude that
\[
0\leq\frac{\widehat{\varepsilon}_{Rz}}{\varepsilon_{Rz}}\leq{\textstyle\sum
\limits_{\substack{gz\neq hz\in\Gamma z \\\phi(g)\neq\phi(h)}}}\frac
{I(Rgz-Rhz)}{I(R\xi_{z})}\leq\ell(G)^{2}C\varepsilon\text{\qquad if }R\geq
R_{0},
\]
which proves the assertion.
\end{proof}

\begin{lemma}
\label{interD}If $(Z_{\ast})$ holds then, for any $g,h\in\Gamma$ such that
$\phi(g)\neq\phi(h)$ and $\gamma\in\Gamma\smallsetminus G,$ we have that%
\[
\int_{\mathbb{R}^{N}}\biggl(\frac{1}{|x|^{\alpha}}\ast\Bigl(|{\textstyle\sum
\limits_{\zeta\in Gz}}\omega_{R\zeta}|^{p}+|{\textstyle\sum\limits_{\zeta\in
Gz}}\omega_{R\gamma\zeta}|^{p}\Bigr)\biggr)\omega_{Rgz}^{p-1}\omega
_{Rhz}=o(\varepsilon_{Rz})
\]
uniformly in $z\in Z.$
\end{lemma}

\begin{proof}
Since $\frac{1}{|x|^{\alpha}}\ast\omega^{p}\in L^{\infty}(\mathbb{R}^{N}) $,
we have that $\frac{1}{|x|^{\alpha}}\ast\Bigl(|{\textstyle\sum\limits_{\zeta
\in Gz}}\omega_{R\zeta}|^{p}+|{\textstyle\sum\limits_{\zeta\in Gz}}%
\omega_{R\gamma\zeta}|^{p}\Bigr)$ is bounded on $\mathbb{R}^{N}$ uniformly in
$z$. Hence,
\begin{align*}
0 &  \leq\int_{\mathbb{R}^{N}}\biggl(\frac{1}{|x|^{\alpha}}\ast
\Bigl(|{\textstyle\sum\limits_{\zeta\in Gz}}\omega_{R\zeta}|^{p}%
+|{\textstyle\sum\limits_{\zeta\in Gz}}\omega_{R\gamma\zeta}|^{p}%
\Bigr)\biggr)\omega_{Rgz}^{p-1}\omega_{Rhz}\\
&  \leq C\int_{\mathbb{R}^{N}}\omega_{Rgz}^{p-1}\omega_{Rhz}=C\int
_{\mathbb{R}^{N}}\omega^{p-1}\omega_{R(hz-gz)}.
\end{align*}
Arguing as in Lemma \ref{epsilongorro}, using this time Lemma \ref{a<1}, we
obtain the conclusion.
\end{proof}

\begin{lemma}
\label{op1}If $Z$ and $V$ satisfy $(Z_{\ast})$ and $(V_{4})$, then
\[
\int_{\mathbb{R}^{N}}V^{+}\sigma_{Rz}^{2}=o(\varepsilon_{Rz})
\]
uniformly in $z\in Z.$
\end{lemma}

\begin{proof}
Let $\kappa>\mu^{\Gamma}(Z)$ be as in assumption $(V_{4})$ (recall that
$V_{\infty}=1$ is assumed). We fix $a>1$ such that $M:=a\mu^{\Gamma}%
(Z)<\min\{2,\kappa\}$. Lemma \ref{inf2} implies that there exists a positive
constant $k_{a}$ such that
\[
I(R\xi_{z})|R\xi_{z}|^{b}e^{a|R\xi_{z}|}\geq k_{a}\qquad\text{if \ }R\geq
\mu_{\Gamma}(Z)^{-1},
\]
where $b:=\frac{N-1}{2}.$ Observing that $M|Rz|=MR=aR\mu^{\Gamma}(Z)\geq
a|R\xi_{z}|$ for all $z\in Z,$ we conclude that
\begin{align*}
\frac{\int_{\mathbb{R}^{N}}V^{+}\sigma_{Rz}^{2}}{\varepsilon_{Rz}} &  \leq
C{\textstyle\sum\limits_{gz\in\Gamma z}}\frac{A(Rgz)}{I(R\xi_{z})}\leq
C{\textstyle\sum\limits_{gz\in\Gamma z}}\frac{A(Rgz)|Rgz|^{b}e^{M|Rgz|}%
}{I(R\xi_{z})|R\xi_{z}|^{b}e^{a|R\xi_{z}|}}\\
&  \leq C{\textstyle\sum\limits_{gz\in\Gamma z}}A(Rgz)|Rgz|^{b}e^{M|Rgz|}%
\qquad\text{if \ }R\geq\mu_{\Gamma}(Z)^{-1},
\end{align*}
where $C$ denotes different positive constants and $A$ is the map defined in
(\ref{A}). Taking Lemma \ref{apositiva} into account, we obtain that
\[
\lim_{R\rightarrow\infty}\frac{\int_{\mathbb{R}^{N}}V^{+}\sigma_{Rz}^{2}%
}{\varepsilon_{Rz}}=0
\]
uniformly in $z\in Z,$ as claimed.
\end{proof}

\begin{lemma}
\label{op2}If $f\in\mathcal{C}_{c}^{0}\left(  \mathbb{R}^{N}\right)  $ and
$q>\max\{\mu^{\Gamma}(Z),1\}$, then
\[
\int_{\mathbb{R}^{N}}f\sigma_{Rz}^{q}=o(\varepsilon_{Rz})
\]
uniformly in $z\in Z.$
\end{lemma}

\begin{proof}
Let us fix $a>1$ such that $\widehat{a}:=\frac{a\mu^{\Gamma}(Z)}{q}<1.$ Lemma
\ref{inf2} yields that there exists $k_{a}>0$ such that
\[
I(R\xi_{z})|R\xi_{z}|^{b}e^{a|R\xi_{z}|}\geq k_{a}\qquad\text{if \ }R\geq
\mu_{\Gamma}(Z)^{-1},
\]
where $b:=\frac{N-1}{2}.$ Since $q\widehat{a}|Rz|=q\widehat{a}R=aR\mu^{\Gamma
}(Z)\geq a|R\xi_{z}|$ for all $z\in Z,$ we conclude that
\begin{align*}
\frac{\int_{\mathbb{R}^{N}}\left\vert f\right\vert \sigma_{Rz}^{q}%
}{\varepsilon_{Rz}} &  \leq C\sum\limits_{gz\in\Gamma z}\frac{\int
_{\mathbb{R}^{N}}\left\vert f\right\vert \omega_{Rgz}^{q}}{I(R\xi_{z})}\leq
C\sum\limits_{gz\in\Gamma z}\frac{\int_{\mathbb{R}^{N}}\left\vert f\right\vert
\omega_{Rgz}^{q}|Rgz|^{b}e^{q\widehat{a}|Rgz|}}{I(R\xi_{z})|R\xi_{z}%
|^{b}e^{a|R\xi_{z}|}}\\
&  \leq C{\textstyle\sum\limits_{gz\in\Gamma z}}\int_{\mathbb{R}^{N}%
}\left\vert f\right\vert \omega_{Rgz}^{q}|Rgz|^{b}e^{q\widehat{a}|Rgz|}%
\qquad\text{if \ }R\geq\mu_{\Gamma}(Z)^{-1},
\end{align*}
where $C$ denote distinct positive constants. Hence, from Lemma \ref{h} we
get
\[
\lim_{R\rightarrow\infty}\frac{\int_{\mathbb{R}^{N}}f\sigma_{Rz}^{q}%
}{\varepsilon_{Rz}}=0
\]
uniformly in $z\in Z.$
\end{proof}

Finally, we need the following result.

\begin{lemma}
\label{razon}Let $\psi:(0,\infty)\rightarrow\mathbb{R}$ be the function given
by
\[
\psi(t):=\frac{a+t+o(t)}{\left(  a+bt+o(t)\right)  ^{\beta}},
\]
where $a>0$, $\beta\in(0,1)$ and $b\beta>1$. Then, there exist constants
$c_{0},t_{0}>0$ such that
\[
\psi(t)\leq a^{1-\beta}-c_{0}t\quad\text{for all }t\in(0,t_{0}).
\]

\end{lemma}

\begin{proof}
Taking $\frac{1}{\beta}<q<b$ and $1<s<r<\beta q,$ we have that there exists
$t_{1}\in(0,1)$ such that
\[
\psi(t)\leq\frac{a+st}{(a+qt)^{\beta}}=\frac{a+rt}{(a+qt)^{\beta}}%
-\frac{(r-s)t}{(a+qt)^{\beta}}\quad\text{for all }t\in(0,t_{1}).
\]
We denote by $f(t):=\frac{a+rt}{(a+qt)^{\beta}}.$ Since $f^{\prime}(0)=\left(
r-\beta q\right)  a^{-\beta}<0,$ there exists $t_{0}\in(0,t_{1}) $ such that
\[
f(t)\leq f(0)=a^{1-\beta}\quad\text{for all }t\in(0,t_{0}).
\]
Consequently,%
\[
\psi(t)\leq a^{1-\beta}-\frac{(r-s)}{(a+q)^{\beta}}t\quad\text{for all }%
t\in(0,t_{0}),
\]
which concludes the proof.
\end{proof}

\smallskip

\begin{proof}
[Proof of Proposition \ref{sub}]Let $\gamma\in\Gamma\smallsetminus G.$ If
$Gz=\{z_{1},\ldots,z_{\ell}\}$ with $\ell:=\ell(G),$ we write
\[
\sigma_{Rz}=\sigma_{Rz}^{1}-\sigma_{Rz}^{2}\text{\qquad with\quad}\sigma
_{Rz}^{1}:={\textstyle\sum\limits_{i=1}^{\ell}}\omega_{Rz_{i}}\text{ \ and
\ }\sigma_{Rz}^{2}:={\textstyle\sum\limits_{i=1}^{\ell}}\omega_{R\gamma z_{i}%
}.
\]
Applying Lemma \ref{l1}\ to $a_{i}:=\omega_{Rz_{i}}(x),$ $\hat{a}_{i}%
:=\omega_{R\gamma z_{i}}(x),$ $b_{i}:=\omega_{Rz_{i}}(y),$ $\hat{b}%
_{i}:=\omega_{R\gamma z_{i}}(y)$ and using Lemma \ref{interD}\thinspace\ we
conclude that
\begin{align*}
\mathbb{D}(\sigma_{Rz}) &  \geq\mathbb{D}(\sigma_{Rz}^{1})+\mathbb{D}%
(\sigma_{Rz}^{2})+o(\varepsilon_{Rz})\\
&  \geq\left\{
\begin{array}
[c]{ll}%
\ell(\Gamma)\mathbb{D}(\omega)+2(p-1)\varepsilon_{Rz}+o(\varepsilon_{Rz}) &
\text{if\ }p>2,\\
\ell(\Gamma)\mathbb{D}(\omega)+4\varepsilon_{Rz}+o(\varepsilon_{Rz}) &
\text{if\ }p=2.
\end{array}
\right.
\end{align*}
Note that, since $\frac{1}{|x|^{\alpha}}\ast\omega^{p}\in L^{\infty
}(\mathbb{R}^{N})$, $\frac{1}{|x|^{\alpha}}\ast|\sigma_{Rz}|^{p}$ is bounded
uniformly in $z$. So, since $\mu^{\Gamma}(Z)<2\leq p$, $\ \chi\Delta\chi
\in\mathcal{C}_{c}^{0}\left(  \mathbb{R}^{N}\right)  $ and $1-\chi^{p}%
\in\mathcal{C}_{c}^{0}\left(  \mathbb{R}^{N}\right)  $, Lemma \ref{op2} yields
that
\[
\int_{\mathbb{R}^{N}}(\chi\Delta\chi)\sigma_{Rz}^{2}=o(\varepsilon
_{Rz})\text{\qquad and\qquad}\int_{\mathbb{R}^{N}}(1-\chi^{p})\left(  \frac
{1}{|x|^{\alpha}}\ast|\sigma_{Rz}|^{p}\right)  \sigma_{Rz}^{p}=o(\varepsilon
_{Rz})
\]
uniformly in $z.$ This, together with Lemmas \ref{descorte},
\ref{epsilongorro} and \ref{op1} and expression (\ref{wI}), yields
\begin{align*}
\Vert\chi\sigma_{Rz}\Vert_{V}^{2} &  \leq\Vert\sigma_{Rz}\Vert^{2}%
+\int_{\mathbb{R}^{N}}V\sigma_{Rz}^{2}-\int_{\mathbb{R}^{N}}(\chi\Delta
\chi)\sigma_{Rz}^{2}\\
&  \leq\ell(\Gamma)\left\Vert \omega\right\Vert ^{2}+\varepsilon_{Rz}%
-\widehat{\varepsilon}_{Rz}+\int_{\mathbb{R}^{N}}V^{+}\sigma_{Rz}%
^{2}+o(\varepsilon_{Rz})\\
&  \leq\ell(\Gamma)\left\Vert \omega\right\Vert ^{2}+\varepsilon
_{Rz}+o(\varepsilon_{Rz}),\\
\mathbb{D}(\chi\sigma_{Rz}) &  \geq\ell(\Gamma)\mathbb{D}(\omega
)+b_{p}\varepsilon_{Rz}+o(\varepsilon_{Rz})-2\int_{\mathbb{R}^{N}}(1-\chi
^{p})\left(  \frac{1}{|x|^{\alpha}}\ast|\sigma_{Rz}|^{p}\right)  \sigma
_{Rz}^{p}\\
&  \geq\ell(\Gamma)\mathbb{D}(\omega)+b_{p}\varepsilon_{Rz}+o(\varepsilon
_{Rz}),
\end{align*}
where $b_{p}:=2(p-1)$ if $p>2$ and $b_{p}:=4$ if $p=2$. Consequently, since
$\left\Vert \omega\right\Vert ^{2}=\mathbb{D}(\omega)$ and $\varepsilon
_{Rz}\rightarrow0$ as $R\rightarrow\infty$ uniformly in $z$, Lemma \ref{razon}
insures that there exist $c_{1},R_{1}>0$ such that
\[
\dfrac{\Vert\chi\sigma_{Rz}\Vert_{V}^{2}}{\mathbb{D}(\chi\sigma_{Rz}%
)^{\frac{1}{p}}}\leq\frac{\ell(\Gamma)\left\Vert \omega\right\Vert
^{2}+\varepsilon_{Rz}+o(\varepsilon_{Rz})}{\left(  \ell(\Gamma)\mathbb{D}%
(\omega)+b_{p}\varepsilon_{Rz}+o(\varepsilon_{Rz})\right)  ^{\frac{1}{p}}}%
\leq\bigl(\ell(\Gamma)\left\Vert \omega\right\Vert ^{2}\bigr)^{\frac{p-1}{p}%
}-c_{1}\varepsilon_{Rz}%
\]
for $R\geq R_{1}$ and $z\in Z.$ Using Lemma \ref{inf2} we conclude that there
exist $c_{0},R_{0}>0$ and $\beta>1$ such that
\[
\dfrac{\Vert\chi\sigma_{Rz}\Vert_{V}^{2}}{\mathbb{D}(\chi\sigma_{Rz}%
)^{\frac{1}{p}}}\leq\bigl(\ell(\Gamma)\left\Vert \omega\right\Vert
^{2}\bigr)^{\frac{p-1}{p}}-c_{0}e^{-\beta R}\quad\text{for any \ }R\geq
R_{0},\text{ }z\in Z,
\]
which is inequality (\ref{lhs}). Finally, since $\pi(\chi\sigma_{Rz}%
)\in\mathcal{N}_{\Omega,V}^{\phi}$ and
\[
J_{V}(\pi(\chi\sigma_{Rz}))=\frac{p-1}{2p}\left(  \dfrac{\Vert\chi\sigma
_{Rz}\Vert_{V}^{2}}{\mathbb{D}(\chi\sigma_{Rz})^{\frac{1}{p}}}\right)
^{\frac{p}{p-1}}<\frac{p-1}{2p}\ell(\Gamma)\left\Vert \omega\right\Vert
^{2}=\ell(\Gamma)c_{\infty},
\]
one has that $c_{\Omega,V}^{\phi}<\ell(\Gamma)c_{\infty}.$
\end{proof}

\smallskip

\begin{proof}
[Proof of Theorem \ref{slna3}]Let $\phi\equiv1,$ so that $\Gamma=G.$ If
assumption $(V_{3})$ holds for $\kappa>\mu_{G},$ we choose $\zeta\in\Sigma$
such that $\mu(G\zeta)\in[\mu_{G},\kappa)$ and set $Z:=G\zeta.$ Thus $\mu
^{G}(Z)=\mu(G\zeta)$ and assumption $(V_{4})$ holds for $\kappa.$ Moreover,
since $\ell(G)\geq3$, $\mu^{G}(Z)=\mu(G\zeta)<2.$ Therefore $(Z_{\ast})$ holds
and we can apply Proposition \ref{sub} to these data to conclude that
$c_{\Omega,V}^{G}<\ell(G)c_{\infty}.$ Corollary \ref{corps}\ then insures that
$J_{V}$ satisfies condition $(PS)_{c}^{G}$ on $\mathcal{N}_{\Omega,V}^{G}$ for
$c:=c_{\Omega,V}^{G}.$ Consequently, there exists $u\in\mathcal{N}_{\Omega
,V}^{G}$ such that $J_{V}(u)=c_{\Omega,V}^{G}.$ Since $\left\vert u\right\vert
\in\mathcal{N}_{\Omega,V}^{G}$ and $J_{V}(\left\vert u\right\vert )=J_{V}(u),$
$\left\vert u\right\vert $ is a positive solution of (\ref{prob}) which is
$G$-invariant and satisfies $J_{V}(\left\vert u\right\vert )<\ell(G)c_{\infty
}.$
\end{proof}

\smallskip

\begin{proof}
[Proof of Theorem \ref{slna4}]If $\phi$ is an epimorphism and $(Z_{0})$ holds,
then $Z\subset\Sigma\smallsetminus\Sigma_{0}$ and $2>\frac{2}{a_{0}}\geq
\mu(Gz)=\mu(\Gamma z).$ Therefore, $\mu^{\Gamma}(Z)<2,$ and hence $(Z_{\ast})$
holds. We choose $R>R_{0}$ and set $d:=\frac{p-1}{2p}\bigl[\bigl(\ell
(\Gamma)\left\Vert \omega\right\Vert ^{2}\bigr)^{\frac{p-1}{p}}-c_{0}%
\varepsilon^{-\beta R}\bigr]^{\frac{p}{p-1}}.$ Proposition \ref{sub}\ then
asserts that the map $\sigma:Z\rightarrow\mathcal{N}_{\Omega,V}^{\phi}\cap
J_{V}^{d}$ given by $\sigma(z):=\pi(\chi\sigma_{Rz})$ is well defined.
Furthermore, $\sigma(gz)=\sigma(z)$ for all $g\in G$\ and\ $\sigma(\gamma
z)=-\sigma(z)$ if $\phi(\gamma)=-1.$ Consequently, $\sigma$ induces a
continuous map $\widehat{\sigma}:Z/G\rightarrow\mathcal{N}_{\Omega,V}^{\phi
}\cap J_{V}^{d},$ given by $\widehat{\sigma}(Gz):=\sigma(z),$ which satisfies
$\widehat{\sigma}((-1)\cdot Gz)=-\widehat{\sigma}(Gz)$ for all $z\in Z.$ This
implies that
\[
\text{genus}(Z/G)\leq\text{ genus}(\mathcal{N}_{\Omega,V}^{\phi}\cap J_{V}%
^{d}).
\]
Since $\mathcal{N}_{\Omega,V}^{\phi}$ is a $\mathcal{C}^{2}$-manifold and
$J_{V}:\mathcal{N}_{\Omega,V}^{\phi}\rightarrow\mathbb{R}$ is an even
$\mathcal{C}^{2}$-function which is bounded from below and satisfies condition
$(PS)_{c}^{\phi}$ on $\mathcal{N}_{\Omega,V}^{\phi}$ for all $c<\ell
(\Gamma)c_{\infty}$, we conclude that$\ J_{V}$ has at least genus$(Z/G)$ pairs
of critical points $\pm u$ with $J_{V}(u)\leq d.$
\end{proof}

\end{document}